\documentclass[11pt]{article}

\usepackage{cite}

\title{Metric stretching and the period map for smooth 4-manifolds}
\author{Christopher Scaduto} 

\date{}

\usepackage[a4paper, total={6in, 8.5in}]{geometry}
\usepackage{graphicx}

\usepackage{times}

\usepackage{amssymb}
\usepackage{tikz-cd}
\usepackage{titling}
\usepackage{mathrsfs} 
\usepackage{booktabs}
\usepackage[all]{xy}
\usepackage{amsthm}
\usepackage{diagbox}
\usepackage{tabularx}
\usepackage{amscd}
\usepackage{caption}
\usepackage{nicefrac}
\usepackage{amsmath}
\usepackage{mathtools}
\usepackage{tikz}
\usepackage{mathabx}
\usepackage{dsfont}
\usepackage{lipsum}
\usepackage{mwe}
\usepackage{slashed}
\usepackage{rotating}
\usepackage{subcaption}
\usepackage[colorlinks,pagebackref,hypertexnames=false]{hyperref} \usepackage[alphabetic,backrefs,msc-links]{amsrefs}
\usepackage{epstopdf,pinlabel}
\definecolor{mint}{HTML}{239B56}
\hypersetup{
    colorlinks=true,
    linkcolor=blue,
    filecolor=magenta,      
    urlcolor=cyan,
    citecolor=mint
}
\usepackage{pgfplots}
\pgfplotsset{compat=newest}
\usetikzlibrary{shapes.geometric}

\usepackage[utf8]{inputenc}

\definecolor{greenish}{rgb}{0.01, 0.75, 0.24}
\definecolor{blueish}{rgb}{0.0, 0.72, 0.92}
\definecolor{orangeish}{rgb}{1.0, 0.55, 0.0}

\newcolumntype{Y}{>{\centering\arraybackslash}X}

\usepackage{sectsty}

\sectionfont{\fontsize{12}{15}\selectfont}

\newcommand{\R}{\mathbb{R}}

\newcommand{\Z}{\mathbb{Z}}

\newcommand{\Q}{\mathbb{Q}}

\newtheorem{theorem}{Theorem}[section]
\newtheorem{prop}[theorem]{Proposition}
\newtheorem{lemma}[theorem]{Lemma}

\newtheorem{corollary}[theorem]{Corollary}
\newtheorem{remark}[theorem]{Remark}
\newtheorem{example}[theorem]{Example}

\newcommand{\Addresses}{{
  \bigskip
  \footnotesize
Christopher Scaduto, \textsc{Department of Mathematics, University of Miami, Coral Gables, FL USA}\par\nopagebreak
  \textit{E-mail address}: \texttt{cscaduto@miami.edu}
}}

\begin{document}

\maketitle

\begin{abstract}
    The period map for a smooth closed $4$-manifold assigns to a Riemannian metric the space of self-dual harmonic $2$-forms. This map is from the space of metrics to the Grassmannian of maximal positive subspaces in the second cohomology, where positivity is defined by cup product. We show that the period map has dense image for every $4$-manifold, and that it is surjective if $b^+=1$. Similar results hold for manifolds of dimension a multiple of four.  The proofs involve families of metrics constructed by stretching along various hypersurfaces. 
\end{abstract}

\vspace{.2cm}

\section{Introduction}
Let $X$ be a smooth, closed, connected, oriented $4$-manifold. The de Rham cohomology $H^2(X)$ has a symmetric bilinear form which integrates wedge products of 2-forms. A subspace $H\subset H^2(X)$ is {\emph{positive}} if the bilinear form on $H$ is positive definite, and is further {\emph{maximal}} if it is not properly contained inside a positive subspace. All maximal positive subspaces have the same dimension, denoted $b^+(X)$. The space of maximal positive subspaces of $H^2(X)$, written $\text{Gr}^+(H^2(X))$, is an open subset of the Grassmannian of all $b^+(X)$-dimensional planes in $H^2(X)$.\\

Let $g$ be a Riemannian metric on $X$. By Hodge theory, the space of $g$-harmonic $2$-forms $\mathcal{H}^2_g(X)$ maps isomorphically to $H^2(X)$ via $\omega\mapsto [\omega]$. The Hodge star $\star_g$ is an involution on $\mathcal{H}^2_g(X)$, and we write $\mathcal{H}^+_g(X)$ for the $(+ 1)$-eigenspace of $\star_g$, the $g$-self-dual $2$-forms, which we identify with its image in $H^2(X)$. Then $\mathcal{H}^+_g(X)$ is a maximal positive subspace of $H^2(X)$. The assignment
\begin{gather}
    \Pi_X:\text{Met}(X)\longrightarrow \text{Gr}^+(H^2(X)) \label{eq:periodmap}\\
    g\longmapsto \mathcal{H}^+_g(X) \nonumber
\end{gather}
is the {\emph{period map}} of $X$. In fact, $\mathcal{H}^+_g(X)$ is unchanged if $g$ is replaced by a conformally equivalent metric, and so $\Pi_X$ descends to the space of conformal classes of metrics.\\

In \cite[Ch. 20]{katz}, it is conjectured that $\Pi_X$ is always surjective. This problem has connections to symplectic topology. Given a symplectic form $\omega$ on $X$, and more generally a near symplectic form, one can construct a metric on $X$ for which $\omega$ is harmonic and self-dual, see \cite[Prop. 1]{adk}. In particular, if $b^+(X)=1$, then $\R\cdot [\omega]\in \text{im}(\Pi_X)$. For a given {\emph{integral}} class $w\in H^2(X)$, Gay and Kirby \cite{gaykirby} construct a near symplectic form $\omega$ with $[\omega]=w$. This implies $\text{im}(\Pi_X)$ is dense when $b^+(X)=1$. Our first result is that this holds more generally.\\

\begin{theorem}\label{thm:main}
     For a smooth, closed, connected, oriented $4$-manifold, $\Pi_X$ has dense image.
\end{theorem}

\vspace{.2cm}

The idea of the proof is as follows. Consider disjoint embedded connected surfaces $\Sigma_i\subset X$ whose Poincar\'{e} duals span a maximal positive subspace $H\subset H^2(X)$. Such $H$ form a dense subset of $\text{Gr}^+(H^2(X))$. Let $Y_i$ be the boundary of a disk-bundle neighborhood of $\Sigma_i$, so that $Y_i$ is a circle bundle over $\Sigma_i$. Consider a metric $g$ on $X$ that is a product metric on pairwise disjoint collar neighborhoods $Y_i\times (-1,1)\subset X$. Form a 1-parameter family of metrics $\{g(r)\}_{r\geq 1}$ on $X$, with $g=g(1)$ which, roughly, stretches along each of the collar neighborhoods. Then
\begin{equation}\label{eq:firstlimit}
    \lim_{r\to \infty}\mathcal{H}^+_{g(r)}=H,
\end{equation}
and the proof follows. A more precise version of \eqref{eq:firstlimit} is given in Proposition \ref{prop:main}, which is proved using standard gluing theory of harmonic forms. \\

Our second result regards the surjectivity of the period map. 

\vspace{.2cm}

\begin{theorem}\label{thm:main2}
     If in addition $b^+(X)=1$, then the period map $\Pi_X$ is surjective.
\end{theorem}

\vspace{.25cm}

The proof of Theorem \ref{thm:main2} builds upon that of Theorem \ref{thm:main}, using higher dimensional families of metrics, parametrized by certain polytopes. The construction of these metric families is adapted from the work of Kronheimer and Mrowka \cite{km-unknot}, further explored by Bloom \cite{bloom}; see also the earlier work \cite{kmos}. Roughly, the families we use are constructed by stretching along boundaries of regular neighborhoods of various surface configurations in $X$. A proper face on a polytope parametrizes metrics that are entirely stretched along some non-empty hypersurface, and, similar to property \eqref{eq:firstlimit}, there are constraints on where $\Pi_X$ can send these faces.\\

In the case that $b^+(X)=1$, $\text{Gr}^+(H^2(X))$ can be identified with hyperbolic space $\mathbb{H}^n$ where $n=b_2(X)-1$. We show that for any (rational) hyperbolic $n$-simplex $\Delta_n$ in $\mathbb{H}^n$, there is a family of metrics parametrized by an $n$-dimensional permutahedron $P_n$ for which the image of the interior of $P_n$ under $\Pi_X$ contains the interior of $\Delta_n$. From this, Theorem \ref{thm:main2} follows.\\

Note that the Grassmannian of maximal positive subspaces $H\subset H^2(X)$ is naturally identified with the Grassmannian of maximal negative subspaces by the assignment $H\mapsto H^\perp$. The period map of the orientation-reversal of $X$ is determined by that of $X$ and this identification. In particular, Theorem \ref{thm:main2} implies that $\Pi_X$ is also surjective if $b^-(X)=1$.\\

The surjectivity of $\Pi_X$ when $b^+(X)=1$ and $X$ is symplectic and minimal (no embedded $2$-spheres of self-intersection $-1$) follows from previous work of Li and Liu \cite{liliu}, who show that under these hypotheses, any $w\in H^2(X)$ with $w^2>0$ is represented by a symplectic form. Without the hypothesis of minimality, their work also implies $\text{im}(\Pi_X)$ is dense when $X$ is symplectic and $b^+(X)=1$. For a proof using similar ingredients, in the case of blow-ups of $\mathbb{C}\mathbb{P}^2$, see \cite[Appendix A]{katz-book}. These proofs use Seiberg--Witten monopoles and pseudo-holomorphic curves. See also Biran \cite[Thm. 3.2]{biran}. Note that our results do not require $X$ to be symplectic, and the proofs we present do not use any gauge theory or pseudo-holomorphic curve theory. \\

Theorems \ref{thm:main} and \ref{thm:main2} hold more generally for $4k$-dimensional manifolds, with essentially the same proofs. See Remarks \ref{rem:1} and \ref{rem:12}.\\

The tracking of period points is central to the wall-crossing phenomena of Donaldson and Seiberg--Witten invariants of closed $4$-manifolds, see \cite{donaldson-irr, kotschick, li-liu-families}. Stretching a metric in a neighborhood of a surface is a main mechanism in proving adjunction inequalities \cite{km-genus,km-embedded}. Further, families of metrics parametrized by polytopes have been important in studying the structure of Floer theories derived from gauge theory; see the already-mentioned works \cite{km-unknot,kmos,bloom}. The author is indebted to the ideas from these works.\\

\vspace{.2cm}

\noindent \textbf{Conformal systoles}

\vspace{.15cm}

\noindent Questions regarding the surjectivity of the period map, and the density of its image, were considered by Katz \cite{katz} in the setting of conformal systoles. Let $X$ be a smooth, closed, oriented manifold of dimension $2n$. The {\emph{conformal $n$-systole}} of $X$ equipped with a Riemannian metric $g$ is defined as
\[
    \text{conf}_{n}(X,g) = \min \left\{  |w|_{g} \;\; \bigl\vert \;\; w\in H^n(X)_\Z \setminus \{0\} \right\}
\]
where $H^n(X)_\Z$ is the image of $H^n(X;\Z)$ in the real cohomology $H^n(X)$, and $|\cdot |_{g}$ is the norm on $H^n(X)$ induced by the $L^2$ norm on $g$-harmonic forms. If $\dim X = 4k$, then $\text{conf}_{2k}(X,g)$ only depends on the image of $g$ under the period map, and continuously so. If $\dim X=2n$, define
\[
    CS(X) = \sup_{g}\;\text{conf}_{n}(X,g)
\]
where the supremum is over all metrics on $X$. The quantity $CS(X)$ for $4$-manifolds was studied in \cite{katz, hamilton}. Generalizing results there, Theorem \ref{thm:main} and its analogues in $\dim=4k$ imply:\\

\begin{corollary}\label{cor:systoles}
    If $\dim X = 4k$, then $CS(X)$ is determined by $H^{2k}(X)_\Z$ with its cup product.
\end{corollary}

\vspace{.4cm}

\noindent \textbf{Outline\;\;} The proof of Theorem \ref{thm:main} is given in Section \ref{sec:background}. This uses Proposition \ref{prop:main}, which describes the behavior of the period map in the limit of cylindrically stretching along a hypersurface. The proof of this proposition is given in Appendix \ref{appendix}. The period map for metric families of polytopes is studied in Section \ref{sec:families}, where Theorem \ref{thm:main2} is proved. \\

\noindent \textbf{Acknowledgements\;\;} The author thanks Nikolai Saveliev for numerous helpful comments. The author was supported by NSF Grant DMS-1952762.\\

\vspace{.4cm}

\section{Decomposing cohomology}\label{sec:background}
We first review some standard material on the intersection pairing, over $\R$, of a $4$-manifold decomposed along a separating $3$-manifold. After stating a precise version of \eqref{eq:firstlimit} in Proposition \ref{prop:main}, we give the proof of Theorem \ref{thm:main}. All manifolds in this paper are oriented and smooth.\\

For a manifold $X$ with boundary, following \cite{aps-i} we write
\[
    \widehat{H}^i(X) = \text{im}\left(H^i(X,\partial X)\to H^i(X)\right), 
\]
which is equivalently the image of $H_c^i(X)\to H^i(X)$, where $H_c^i(X)$ is de Rham cohomology with compact supports on the {\emph{interior of $X$}}. This last bit of notation is non-standard: what we write as $H^i_c(X)$ would usually be written as $H^i_c(\text{int}(X))$. We also write $\widehat{b}_i(X) = \dim \widehat{H}^i(X)$. \\

Let $X$ be a closed $4$-manifold, and $Y\subset X$ a closed $3$-manifold separating $X$ into two pieces,
\begin{equation}\label{eq:xdecomp}
X = X_1\cup_Y X_2,
\end{equation}
where $X_1,X_2$ are compact with $\partial X_1= Y = -\partial X_2$. We do not assume that any of these manifolds are connected. Consider the Mayer--Vietoris sequence
\begin{equation}
\begin{tikzcd}[]\label{eq:mv}
    \cdots \;\; H^1(Y) \arrow[r, "\delta"] & H^2(X) \arrow[r, "j"] & H^2(X_1)\oplus H^2(X_2) \arrow[r, "k"] & H^2(Y)\;\; \cdots
\end{tikzcd}
\end{equation}
The map $\delta$ has the following description in de Rham cohomology. Let $[\alpha]\in H^1(Y)$, and choose a neighborhood $Y\times (-1,1)\subset X$. Let $\rho\in C^\infty(\R)$ have $\int \rho(t)dt=1$ and $\text{supp}(\rho)\subset(-1,1)$. Then
\begin{equation}\label{eq:deltamap}
    \delta[\alpha] = [ \rho(t)\alpha\wedge dt], 
\end{equation}
where the form $\rho(t)\alpha\wedge dt$ is defined on all of $X$ via extension by zero. Furthermore,
\begin{equation}\label{eq:dualrel}
    \int_X  \omega \wedge \delta(\alpha) = \int_Y \iota^\ast(\omega)\wedge \alpha
\end{equation}
for all closed forms $\omega\in \Omega^2(X)$, where $\iota:Y\to X$ is the inclusion map.

\vspace{.3cm}

\begin{lemma}\label{lemma:mv1}
    $b_2(X) = \widehat{b}_2(X_1)+ \widehat{b}_2(X_2) + 2\dim ({\rm{im}}(\delta))$.
\end{lemma}

\vspace{.07cm}

\begin{proof}
    For brevity, write $|V|=\dim V$. Exactness of \eqref{eq:mv} gives 
    \begin{equation}\label{eq:mvproof1}
        b_2(X) = b_2(X_1)+b_2(X_2) + |\text{im}(\delta)| - |\text{im}(k)|.
    \end{equation}
    Write $\iota_i:Y\to X_i$ for the inclusion maps. Then $\text{im}(k) =\text{im}(\iota^\ast_1)+ \text{im}(\iota^\ast_2)$. Further,
    \begin{equation}\label{eq:mvproof2}
        |\text{im}(k)| + |\text{im}(\iota^\ast)| = |\text{im}(\iota_1^\ast)| + |\text{im}(\iota_2^\ast)|,
    \end{equation}
    as follows from $\text{im}(\iota^\ast_1)\cap \text{im}(\iota^\ast_2)=\text{im}(\iota^\ast)$. We also have, for $i=1,2$,
    \begin{equation}\label{eq:mvproof3}
        |\text{im}(\iota_i^\ast)| = b_2(X_i) - \widehat{b}_2(X_i),
    \end{equation}
    which follows from the long exact sequence of the pair $(X_i,Y)$. Finally, we claim
     \begin{equation}
    |\text{im}(\iota^\ast)|=|\text{im}(\delta)|. \label{eq:mvproof4}
    \end{equation}
    Let $l:H^1(Y)\to \text{im}(\iota^\ast)^\ast$ send $[\alpha]$ to the linear form $l_\alpha(\beta) = \int_Y \alpha \wedge \beta$, restricted to $[\beta]\in\text{im}(\iota)$. This map is surjective, by Poincar\'{e} duality. By \eqref{eq:dualrel}, $l_\alpha=0$ if and only if $\int_X \delta(\alpha)\wedge \omega=0$ for all $[\omega]\in H^2(X)$, which by Poincar\'{e} duality happens if and only if $\delta[\alpha]=0$. Thus $l$ induces $H^1(Y)/\text{ker}(\delta)\cong \text{im}(\iota^\ast)^\ast$, implying \eqref{eq:mvproof4}. The lemma now follows from \eqref{eq:mvproof1}--\eqref{eq:mvproof4}.
\end{proof}

\vspace{.2cm}

Next, $H_c^2(X_i)$ has its own pairing, again induced by integration and wedge product. This pairing is, in general, degenerate. The induced pairing on $\widehat{H}^2(X_i)$ is non-degenerate, by Poincar\'{e} duality. Write $b^\pm(X_i)$ for the dimensions of maximal $\pm$-subspaces in $\widehat{H}^2(X_i)$.

\vspace{.2cm}

\begin{lemma}\label{lemma:mv2}
    $b^\pm(X) = b^\pm(X_1)+ b^\pm(X_2) + \dim ({\rm{im}}(\delta))$.
\end{lemma}

\vspace{.0cm}

\begin{proof} Let $H_i\subset H^2_c(X_i)$ be a subspace which maps isomorphically to $\widehat{H}^2(X_i)$. In particular, the pairing on $H_i$ is non-degenerate. We have a natural injection $H_1\oplus H_2\oplus \text{im}(\delta)\to H^2(X)$, which we treat as an inclusion, and the pairing is orthogonal with respect to the decomposition. Indeed, forms in $H_1$ have support disjoint from those in $H_2$, and both pair trivially with $\text{im}(\delta)$ by \eqref{eq:dualrel}. From \eqref{eq:deltamap}, the pairing is trivial on $\text{im}(\delta)$. Diagonalize the form so that on this subspace it is
\[
    \underbrace{\langle +1 \rangle^{b^+(X_1)} \oplus \langle -1\rangle^{b^-(X_1)}}_{\text{ on  }\;H_1} \oplus \underbrace{\langle +1 \rangle^{b^+(X_2)} \oplus \langle -1\rangle^{b^-(X_2)}}_{\text{ on  }\;H_2} \oplus  \underbrace{\langle 0 \rangle^{\dim(\text{im}(\delta))} }_{\text{ on  }\;\text{im}(\delta)}
\]
By basic algebra of non-degenerate symmetric bilinear forms over $\R$ and Lemma \ref{lemma:mv1}, there exists $W\subset H^2(X)$ such that $H^2(X)=H_1\oplus H_2\oplus \text{im}(\delta)\oplus W$, the complement of $W$ with respect to the pairing is $W^\perp = H_1\oplus H_2 \oplus W$, and on $\text{im}(\delta)\oplus W$ it is equivalent to a sum of hyperbolic planes,
\[
     \left(\begin{array}{cc} 0 & 1 \\ 1 & 0\end{array}\right).
\]
Each hyperbolic plane is equivalent over $\R$ to $\langle +1\rangle\oplus \langle -1\rangle$. The result follows.
\end{proof}

\vspace{.2cm}

\noindent A corollary of Lemma \ref{lemma:mv2} is the well-known additivity of the signature for $4$-manifolds.\\

As above, let $X_1$ and $X_2$ be compact $4$-manifolds with $\partial X_1 = Y = -\partial X_2$. For $r\geq 0$, let
\begin{equation}\label{eq:maindecomp}
    X(r) = X_1\cup Y\times [-r,r] \cup X_2
\end{equation}
where $Y\times \{-r\}$ and $Y\times \{r\}$ are glued to the boundaries of $X_1$ and $X_2$, respectively. Set $X=X(1)$. (This slight deviation from above is to match our convention in the appendix.) Define a metric $h(r)$ on $X(r)$ by $h(r)|_{Y\times [-r,r]}=g_Y + dt^2$ for some fixed metric $g_Y$ on $Y$, and  $h(r)|_{X_i}=g_i$ for some fixed metrics $g_i$ on $X_i$ which agree with $g_Y+dt^2$ in collar neighborhoods of their boundaries. There are diffeomorphisms $f_r:X\to X(r)$ which are natural up to isotopy. We say that the 1-parameter family of metrics $g(r) = f_r^\ast h(r)$ on $X$ stretches along $Y$ in a cylindrical fashion.\\

Let $X_1(\infty)=X_1\cup Y \times [0,\infty)$ where $Y\times \{0\}$ is glued to the boundary of $X_1$. This comes with a metric $g_1(\infty)$, equal to $g_1$ on $X_1$ and $g_Y + dt^2$ on the cylinder. Let $\mathcal{H}^2_{X_1}$ denote the $L^2$ harmonic $2$-forms on $X_1(\infty)$, where the $L^2$ metric is defined by $g_1(\infty)$. There is a natural identification
\[
    \mathcal{H}^2_{X_1} = \widehat{H}(X_1),
\]
due to Atiyah--Patodi--Singer \cite[Prop. 4.9]{aps-i}. The self-dual $L^2$ harmonic $2$-forms $\mathcal{H}^+_{X_1}$ give a maximal positive subspace for the pairing on $\widehat{H}(X_1)$. Similar remarks hold for $X_2$.\\

Recall that ${\rm{Gr}}^+(H^2(X))$ is an open subset of the Grassmannian of $b^+(X)$-dimensional planes in $H^2(X)$. We write $\overline{{\rm{Gr}}}^+(H^2(X))$ for its closure in this ambient Grassmannian. Equivalently, this is the space of maximal semi-positive subspaces in $H^2(X)$.\\

\begin{prop}\label{prop:main}
    Let $g(r)$ be a 1-parameter family of metrics which as $r\to \infty$ stretches along a $3$-manifold $Y\subset X$ in a cylindrical fashion as described above. Then, in $\overline{{\rm{Gr}}}^+(H^2(X))$, we have
    \begin{equation}\label{eq:propmain}
        \lim_{r\to \infty} \mathcal{H}^+_{g(r)} = j^{-1}\left(\mathcal{H}^+_{X_1} \oplus   \mathcal{H}^+_{X_2} \right)
    \end{equation}
    where $j:H^2(X)\to H^2(X_1)\oplus H^2(X_2)$ is the map in the Mayer--Vietoris sequence.\\
\end{prop}

\vspace{.1cm}

The proof of Proposition \ref{prop:main} follows from standard gluing theory of harmonic forms. For completeness, we explain in Appendix \ref{appendix} how the result follows from the work of \cite{clmi}.\\

Another way to state \eqref{eq:propmain} is as follows. For $i=1,2$, choose subspaces $H^+_i\subset H_c^2(X_i)$ which map isomorphically to $\mathcal{H}^+_{X_i}\subset \widehat{H}^2(X_i)$. Then, identifying $H^+_i$ with its image in $H^2(X)$, we have
\begin{equation}\label{eq:propeqalt}
\lim_{r\to \infty} \mathcal{H}^+_{g(r)} = H^+_1 + \text{im}(\delta) + H^+_2,
\end{equation}
and this is a direct sum of subspaces. Note that the possible choices $H^+_i\subset H^2(X)$ differ by the addition of elements in $\text{im}(\delta)$, so that \eqref{eq:propeqalt} is independent of the choices, as expected.\\

Proposition \ref{prop:main} is the main tool used to prove Theorem \ref{thm:main}. A simple case is when $\text{im}(\delta)=0$, so that $j:H^2(X)\to H^2(X_1)\oplus H^2(X_2)$ is an isomorphism, and \eqref{eq:propmain} may be written as
\[
        \lim_{r\to \infty} \mathcal{H}^+_{g(r)} =\mathcal{H}^+_{X_1} \oplus   \mathcal{H}^+_{X_2}.
\]
Another simplification occurs when $X_1$ and $X_2$ are definite, for in this case the right side of \eqref{eq:propmain} is independent of metrics. Both of these simplifications are present when we apply this proposition to prove Theorem \ref{thm:main}. Before doing so, we illustrate Proposition \ref{prop:main} with two simple examples.\\

\begin{example}\label{ex1}{\emph{
Let $X=\mathbb{C}\mathbb{P}^2\# \overline{\mathbb{C}\mathbb{P}}^2$. Let $H\in H^2(X)$ be dual to a hyperplane, and $E\in H^2(X)$ to an exceptional sphere: $H^2=1$, $E^2=-1$, $H\cdot E=0$. Represent $H^2(X)$ as $\R^2$, with $x$-axis $\R\cdot E$, and $y$-axis $\R\cdot H$, see Figure \ref{fig:cp2cp2barintegral}. The positive subspaces are lines with slope $s$ where $|s|>1$ (including $|s|=\infty$). Stretching a metric along the connected sum $3$-sphere, we limit to the $y$-axis.
}}
\end{example}

\vspace{.1cm}

\begin{example}\label{ex2}{\emph{
Let $X=S^2\times S^2$. The duals of $S^2\times \{pt\}$ and $\{pt\}\times S^2$ in $H^2(X)$ give a basis for the $x$- and $y$-axes, see Figure \ref{fig:s2s2integral}. The positive subspaces are lines with positive slope. Decomposing $X$ along $S^1\times S^2$ and $S^2\times S^1$, and stretching in a cylindrical fashion, we limit to the $y$-axis and $x$-axis, respectively. In each case, the axis is the line $\text{im}(\delta)\subset H^2(X)$. 
}} 
\end{example}

\vspace{.1cm}

\begin{figure}
\centering
\begin{minipage}{.5\textwidth}
  \centering
  \begin{tikzpicture}[scale=0.65]
    \coordinate (Origin)   at (0,0);
    \coordinate (XAxisMin) at (-5,0);
    \coordinate (XAxisMax) at (5,0);
    \coordinate (YAxisMin) at (0,-5);
    \coordinate (YAxisMax) at (0,5);
    \draw [thin, gray,-latex] (XAxisMin) -- (XAxisMax);
    \draw [thin, gray,-latex] (YAxisMin) -- (YAxisMax);
    \clip (-5,-5) rectangle (5cm,5cm); 
    \coordinate (Bone) at (0,2);
    \coordinate (Btwo) at (2,-2);
	    \foreach \x in {1,...,17}{
		\draw[thick,gray] (Origin) -- (45+5*\x:6cm) {};
		\draw[thick,gray] (Origin) -- (45+5*\x:-6cm) {};
    }
	\draw [thin, dashed, gray] (Origin) -- (45:6cm);
	\draw [thin, dashed, gray] (Origin) -- (135:6cm);
	\draw [thin, dashed, gray] (Origin) -- (45:-6cm);
	\draw [thin, dashed, gray] (Origin) -- (135:-6cm);
    \foreach \x in {-7,-6,...,5}{
      \foreach \y in {-7,-6,...,5}{
        \node[draw,circle,inner sep=1pt,fill] at (2*\x,2*\y) {};
      }
    }
  \end{tikzpicture}
  \captionof{figure}{$\mathbb{C}\mathbb{P}^2\# \overline{\mathbb{C}\mathbb{P}}^2$}
  \label{fig:cp2cp2barintegral}
\end{minipage}%
\begin{minipage}{.5\textwidth}
  \centering
  \begin{tikzpicture}[scale=0.65]
    \coordinate (Origin)   at (0,0);
    \coordinate (XAxisMin) at (-5,0);
    \coordinate (XAxisMax) at (5,0);
    \coordinate (YAxisMin) at (0,-5);
    \coordinate (YAxisMax) at (0,5);
    \draw [thin, gray,-latex] (XAxisMin) -- (XAxisMax);
    \draw [thin, gray,-latex] (YAxisMin) -- (YAxisMax);
    \clip (-5,-5) rectangle (5cm,5cm); 
    \coordinate (Bone) at (0,2);
    \coordinate (Btwo) at (2,-2);
	    \foreach \x in {1,...,17}{
		\draw[thick,gray] (Origin) -- (5*\x:6cm) {};
		\draw[thick,gray] (Origin) -- (5*\x:-6cm) {};
    }
    \foreach \x in {-7,-6,...,5}{
      \foreach \y in {-7,-6,...,5}{
        \node[draw,circle,inner sep=1pt,fill] at (2*\x,2*\y) {};
      }
    }
  \end{tikzpicture}
  \captionof{figure}{$S^2\times S^2$}
  \label{fig:s2s2integral}
\end{minipage}
\end{figure}

\begin{proof}[Proof of Theorem \ref{thm:main}]
    Let $H\in \text{Gr}^+(H^2(X))$ be spanned by pairwise orthogonal rational classes $\sigma_1,\ldots,\sigma_n\in H^2(X;\Q)$, where $n=b^+(X)$. Such $H$ are dense in $\text{Gr}^+(H^2(X))$. Choose $N\in \Z_{+}$ such that $N\sigma_i\in H^2(X)_\Z$ for all $i$. Let $\Sigma_i\subset X$ be a closed oriented connected surface such that $[\Sigma_i]\in H_2(X;\Z)/\text{torsion}$ is Poincar\'{e} dual to $N\sigma_i$. Here we use that every class in $H_2(X;\Z)$, for a closed $4$-manifold, is represented by an embedded closed surface. Since $\Sigma_i\cdot\Sigma_j=0$ for $i\neq j$, a well-known procedure (by adding handles) allows us to further arrange that $\Sigma_i\cap \Sigma_j = \emptyset$ for $i\neq j$. For background on the facts used here, see e.g. \cite[Ch. 1, 2]{gs}.\\

    Let $W_i\subset X$ be a closed disk-bundle neighborhood of $\Sigma_i$. We may assume $W_i\cap W_j=\emptyset $ for $i\neq j$. Let $Y_i=\partial W_i$, which is a circle bundle over $\Sigma_i$ with non-zero Euler class $e_i\in H^2(\Sigma_i)$. The Gysin exact sequence for the circle bundle $\pi:Y_i\to \Sigma_i$ is given by
    \begin{equation}
\begin{tikzcd}[]\label{eq:gysin}
    \cdots \;\; H^1(Y_i) \arrow[r, "\pi_\ast"] & H^0(\Sigma_i) \arrow[r, "e_i\wedge "] & H^2(\Sigma_i)  \arrow[r, "\pi^\ast"] & H^2(Y_i)\;\; \cdots
\end{tikzcd}
\end{equation}
and thus $\pi^\ast:H^2(Y_i)\to H^2(\Sigma_i)$ is zero. This implies the vanishing of the restriction map
\begin{equation}\label{eq:restmapinproof}
    H^2(W_i)\to H^2(Y_i),
\end{equation}
as $W_i$ deformation retracts onto $\Sigma_i$. Now let $X_2=W_1\cup \cdots \cup W_n$, $Y=\partial X_2 = Y_1\cup \cdots \cup Y_n$, and let $X_1$ be the closure of $X\setminus X_2$. Construct a family of metrics $g(r)$ as in Proposition \ref{prop:main}, stretching along $Y$ in a cylindrical fashion. Note the vanishing of each map in \eqref{eq:restmapinproof} implies that $\iota^\ast_2:H^2(X_2)\to H^2(Y)$ is zero, which in turn implies $\text{im}(\delta)\cong \text{im}(\iota^\ast)$ is zero. Then \eqref{eq:propmain} yields
\[
    \lim_{r\to \infty} \mathcal{H}^+_{g(r)} = \widehat{H}^2(X_2) = \bigoplus_{i=1}^n \R \cdot \sigma_i = H.
\]
We have used that $X_1$ is negative definite and $X_2$ is positive definite, with each $\widehat{H}^2(W_i)=H^2(W_i)$ generated by the Poincar\'{e} dual of $[\Sigma_i]$. This completes the proof.
\end{proof}

\vspace{0.1cm}

\begin{remark}\label{rem:1}{\emph{
For a closed, connected manifold $X$ of dimension $4k$, the period map is defined as a map $\Pi_X:\text{Met}(X)\to \text{Gr}^+(H^{2k}(X))$. The above proof easily adapts to this case. The surfaces $\Sigma_i$ are now $2k$-dimensional. To obtain such surfaces we appeal to a classic result of Thom, which says that for every class $\sigma\in H_i(X;\Z)$, there is some $n\in \Z_+$ such that $n\sigma$ is represented by an embedded submanifold. All other aspects of the proof are straightforward adaptations.
}}
\end{remark}

\vspace{0.051cm}

\begin{remark}\label{rem:2}{\emph{
We may also allow $X$ to have boundary. Consider $\text{Met}_{\text{cyl}}(X)$, the space of metrics on $\text{int}(X)$ which are conformally equivalent near the boundary to cylindrical end product metrics. The period map is then $\Pi_X:\text{Met}_{\text{cyl}}(X)\to \text{Gr}^+(\widehat{H}^{2}(X))$. The proof above adapts to show if $X$ is connected, then $\Pi_X$ has dense image.
}}
\end{remark}

\vspace{0.1cm}

The proof method of Theorem \ref{thm:main} also shows the following. Let $H$ be a rational maximal {\emph{semi}}-positive definite subspace of $H^2(X)$, which is not positive. In particular, we have
\[
    H\in \partial \overline{\text{Gr}}^+(H^2(X)).
\]
We can find surfaces $\Sigma_i$ as in the above proof, but now only $\Sigma_i\cdot \Sigma_i\geq 0$, with equality holding for certain $i$. Constructing the family of metrics $g(r)$ as before still gives $\lim_{r\to \infty}\mathcal{H}^+_{g(r)} = H$.

\vspace{.3cm}

\section{Metric families and hyperbolic polyhedra}\label{sec:families}
In this section we prove Theorem \ref{thm:main2}. As motivation, first consider the case in which $b_2(X)=1$ and $b^+(X)=1$, such as in Examples \ref{ex1} and \ref{ex2}. Proving surjectivity of $\Pi_X$ is straightforward in this case, as the positive Grassmannian is simply an interval. Choose two non-zero vectors $v_1$ and $v_2$ in $H^2(X)_\Z$ which square to zero and whose spans in $H^2(X)$ give the two endpoints of the compactified Grassmannian. Let $\Sigma_1$ and $\Sigma_2$ be surfaces in $X$ Poincar\'{e} dual to $v_1$ and $v_2$, respectively. Form a family of metrics $g(r)$ parametrized by $\R$ which stretches in a cylindrical fashion along the boundary of a regular neighborhood of $\Sigma_1$ (resp. $\Sigma_2$) as $r\to -\infty$ (resp. $+\infty$). An application of Proposition \ref{prop:main} shows that the period map on this family induces a map from an interval to an interval which is bijective on endpoints, and is thus surjective.\\

Consider next the case $b_2(X)=2$ and $b^+(X)=1$. Here the positive Grassmannian may be identified with the hyperbolic plane $\mathbb{H}^2$. In this case we will construct $2$-dimensional families of metrics, parametrized by polygons, where each edge consists of metrics entirely stretched along some $3$-manifold. Proposition \ref{prop:main} gives constraints on where the boundaries of these polygons map to in $\mathbb{H}^2$. We can then show that there is a polygonal metric family which maps onto any given (rational) hyperbolic simplex in $\mathbb{H}^2$, implying the surjectivity of $\Pi_X$. In the general case where $b^+(X)=1$ and $b_2(X)=n+1$, we will show there is an $n$-dimensional polytope (the permutahedron $P_n$) of metrics which surjects onto any given (rational) hyperbolic simplex in $\mathbb{H}^n$. \\

In Subsection \ref{subsec:families}, we review a construction for families of metrics parametrized by polyhedra, following Kronheimer and Mrowka \cite[\S 3.9]{km-unknot} and Bloom \cite{bloom}, and study the behavior of the period map on the faces of these polyhedra. In Subsection \ref{subsec:proofsurjectivity} we prove Theorem \ref{thm:main2}. In Subsection \ref{subsec:examples} we give examples which illustrate the constructions in the setting of the hyperbolic plane $\mathbb{H}^2$.\\

\subsection{Polytopes and the period map}\label{subsec:families}

Let $X$ be a closed and connected $4$-manifold. Suppose we have a collection
\[
    \mathcal{C} = \{ Y_1, Y_2,\ldots, Y_m \}
\]
of closed embedded $3$-manifolds $Y_i\subset X$ that are pairwise disjoint. Each $Y_i$ may be disconnected. To define an $m$-dimensional model family of metrics, where each parameter stretches in a cylindrical fashion along each $Y_i$, first choose disjoint collar neighborhoods $Y_i\times [-1,1] \subset X$, and a metric $g_{\mathbf{0}}$ on $X$ equal to $g_{Y_i}+ds^2$ on each such collar. Consider a family of smooth functions $\psi_r:\R\to \R$, depending smoothly on $r\in [0,\infty]$, such that $\psi_0(s)=1$ for all $s$, $\psi_r(s)=1$ for $|s|\geq 1$, and
\[
    \psi_r(s) = \frac{1+1/r^2}{s^2+1/r^2}
\]
for all $s$ in some neighborhood $U\subset (-1,1)$ of $0$, with $\psi_r$ uniformly bounded on $\R\setminus U$. Note $\psi_\infty(s)=1/s^2$ near $0$. For $\mathbf{r}=(r_1,\ldots,r_m)\in [0,\infty)^m$ define a smooth metric $g_{\mathbf{r}}$ on $X$ by 
\begin{equation}\label{eq:firstmodel}
    g_{\mathbf{r}}|_{Y_i\times [-1,1]} = g_{Y_i} + \psi_{r_i}(s)ds^2
\end{equation}
for each $i=1,\ldots,m$, and $g_{\mathbf{r}}=g_{\mathbf{0}}$ on the complement of the collar neighborhoods. This is similar to the situation described in \eqref{eq:maindecomp}, where $Y:=Y_1\cup\cdots\cup Y_m$, except that we use a different stretching parameter for each $Y_i$, and also, we do not assume $Y$ separates $X$ into $X_1$ and $X_2$.\\

We call the $g_{\mathbf{r}}$ for ${\mathbf{r}}\in [0,\infty]^m$ constructed above {\emph{broken metrics}}. If ${\mathbf{r}}\in [0,\infty)^m$, then $g_{\mathbf{r}}$ is an honest smooth metric on $X$. The broken metric for $\mathbf{r}=(\infty,\ldots,\infty)$, restricted to
\[
    X\setminus Y,
\]
is conformally equivalent to a metric with cylindrical end metrics on $(-1,0)\times Y$ and $(0,1)\times Y$, and we say that this metric is {\emph{broken along $\mathcal{C}$}}. More generally, any $g_\mathbf{r}$ for ${\mathbf{r}}\in [0,\infty]^m$ is broken along some subcollection of $\mathcal{C}$, determined by those $i$ for which $r_i=\infty$.\\

Consider a metric $\mathbf{g}$ broken along $\mathcal{C}=\{Y_1,\ldots,Y_m\}$ as constructed above. We define
\begin{equation*}\label{eq:extofperiodmap}
    \mathcal{H}^+_{\mathbf{g}}(X) := j^{-1}\left(  \mathcal{H}^+_{X\setminus Y} \right)
\end{equation*}
where $j:H^2(X)\to H^2(X\setminus Y)$ is the map induced by inclusion, and 
\[  
    \mathcal{H}^+_{X\setminus Y}\subset \widehat{H}^2(X\setminus Y) = \text{im}\left(H^2_c(X\setminus Y)\to H^2(X\setminus Y)\right)
\]
is the space of self-dual harmonic $2$-forms on $X\setminus Y$ with cylindrical end metric $\mathbf{g}|_{X\setminus Y}$. We may now extend the definition of the period map $\Pi_X$ to this broken metric by setting $\Pi_X(\mathbf{g})=\mathcal{H}^+_{\mathbf{g}}(X)$. With this definition, a minor variation of Proposition \ref{prop:main} implies the following.\\

\begin{prop}\label{prop:extperiodmap}
    Consider a disjoint collection of $m$ hypersurfaces $\mathcal{C}$ in the closed $4$-manifold $X$, and construct an associated model family of broken metrics $g_\mathbf{r}$ as above, where $\mathbf{r}\in [0,\infty]^m$. Then the period map extends to define a continuous map on this family:
    \begin{gather*}
         [0,\infty]^m \longrightarrow \overline{{\rm{Gr}}}^+(H^2(X)),\\
         \mathbf{r} \longmapsto \Pi_X(g_{\mathbf{r}}).
    \end{gather*}
\end{prop}

\vspace{.3cm}

Next, consider an $n$-dimensional convex polytope $P$. We require that to each codimension $m$ face $F$ on the boundary of $P$ is associated  $\mathcal{C}_F$, a collection of $m$ pairwise disjoint hypersurfaces. The intersection of two faces $F$ and $F'$ is another face, and we require
\[
    \mathcal{C}_{F\cap F'} \subset  \mathcal{C}_F \cup \mathcal{C}_{F'}.
\]
We say that $P$ is {\emph{labelled by hypersurfaces in $X$}} if these conditions are satisfied. We construct
\[
    G_P = \{ g_p \}_{p\in P},
\]
a family of broken metrics parametrized by $P$, as follows. If $\dim P =0$, then $G_P$ is any smooth metric. Suppose $\dim < n$ families are constructed, and $\dim P=n$. A point $p$ on the interior of a proper face of $P$ has a neighborhood of the form $(0,\infty]^{m}\times P'$ where $P'$ is a convex polytope and $m>0$. To describe a neighborhood of $g_p$ in $G_P$, we proceed as in \eqref{eq:firstmodel}, but in place of a fixed $g_{\mathbf{0}}$, we apply the construction to the entire family $G_{P'}$. This gives a family of broken metrics parametrized by $[0,\infty]^{n-m}\times P'$, and is our model for a neighborhood of $g_p\in G_P$. Using convexity properties of metrics we can construct a family parametrized by a collar neighborhood of $\partial P\subset P$ satisfying these model conditions, then extend this to all of $P$, defining $G_P$.\\

The family of metrics $G_P$ parametrized by the polytope $P$ gives a map
\[
     P \longrightarrow \overline{{\rm{Gr}}}^+(H^2(X))
\]
by sending $p\in P$ to $\Pi_X(g_p)$, and this map is continuous by Proposition \ref{prop:extperiodmap}. Furthermore, as $g_p$ for $p\in\text{int}(P)$ is a smooth (unbroken) metric, the image of $\text{int}(P)$ under $\Pi_{X,P}$ is contained in the image of the period map $\Pi_X$ from \eqref{eq:periodmap}.\\

\begin{figure}[t]
\centering
\includegraphics[scale=0.9]{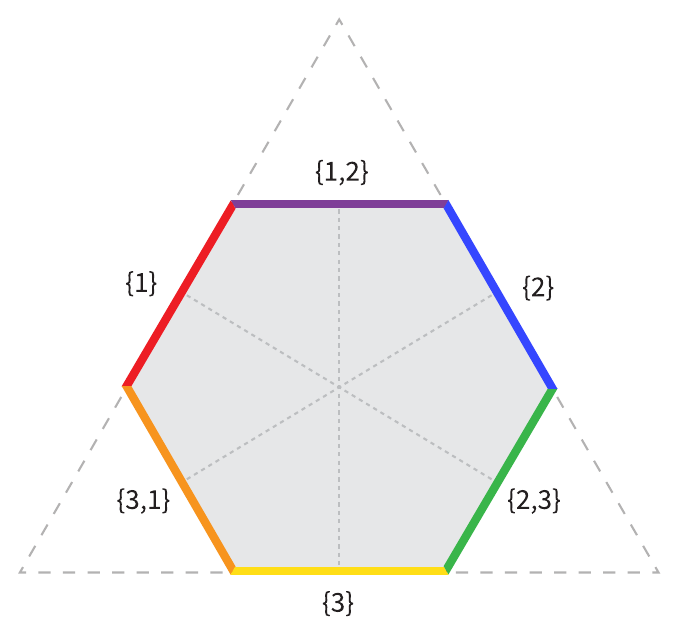}
\caption{\small{To three surfaces $\Sigma_1,\Sigma_2,\Sigma_3\subset X$ is associated a metric family parametrized by the permutahedron $P_2$, which is a hexagon. Each face is labelled by a non-empty proper subset of $\{1,2,3\}$ (a nested sequence with one subset). Faces $\{i\}$ and $\{i,j\}$ correspond to metrics broken along boundaries of regular neighborhoods of $\Sigma_i$ and $\Sigma_i\cup \Sigma_j$, respectively. The vertex joining $\{i\}$ and $\{i,j\}$ is the nested sequence $\{i\}\subset \{i,j\}$. The hexagon is made up of six wedges, each a model family of stretching metrics $[0,\infty]^2$. The permutahedron $P_n$ can be viewed as a truncation of the simplex $\Delta_n$, as indicated here.}}\label{fig:hexagon}
\end{figure}

Now suppose we are given a collection of closed, embedded surfaces $\mathcal{S}=\{\Sigma_1,\ldots,\Sigma_{n+1}\}$ in $X$. We do not assume that the surfaces are pairwise disjoint. For each $I\subset \{1,\ldots,n+1\}$ let $Y_I\subset X$ be the $3$-manifold which is the boundary of a regular neighborhood $X_I$ of $\cup_{i\in I} \Sigma_i$. We make these choices so that $I\subset J$ implies $Y_I\cap Y_J = \emptyset$. There is a polytope, the $n$-dimensional permutahedron $P_{n}$, whose proper faces are labelled by nested sequences $\mathbf{I}=\{I_i\}_{i=1}^l$ ($1\leq l\leq n$) of the form
\begin{equation}\label{eq:nestedseq}
    \emptyset\subsetneq  I_1\subsetneq I_2 \subsetneq \cdots \subsetneq  I_l \subsetneq \{1,\ldots,n+1\}.
\end{equation}
Such nested sequences form a poset by declaring $\mathbf{I}\leq \mathbf{I}'$ if and only if $\mathbf{I}'\subset \mathbf{I}$. Then the poset of proper non-empty faces of $P_n$ is isomorphic to the poset of such nested sequences. The face $F_\mathbf{I}\subset P_{n}$ corresponding to $\mathbf{I}=\{I_i\}_{i=1}^l$ has codimension $l$ in $P_n$. Now $P_n$ can be viewed as a polytope with faces labelled by hypersurfaces in $X$: the face $F_\mathbf{I}$ has the associated collection
\[
    \mathcal{C}_{\mathbf{I}} = \{ Y_{I} \; \mid \; I\in \mathbf{I} \}
\]  
which by construction is a disjoint set of closed embedded $3$-manifolds in $X$. Applying the metric family construction, to any collection of surfaces $\mathcal{S}=\{\Sigma_1,\ldots,\Sigma_{n+1}\}$ in $X$ we obtain family of metrics on $X$ parametrized by the $n$-dimensional permutahedron $P_n$. See for example Figure \ref{fig:hexagon}. We remark that for a nested sequence as in \eqref{eq:nestedseq}, by convention we define $I_{l+1}=\{1,\ldots,n+1\}$. These permutahedral families of metrics fit into a more general framework of metric families parametrized by graph associahedra, explored by Bloom \cite{bloom}. We warn the reader that our notation $P_n$ is non-standard, in that $n$ refers to the dimension of the permutahedron. \\

Now, given any subset $V=\{v_1,\ldots,v_{n+1}\}\subset H^2(X)_\mathbb{Z}$, we can choose closed, connected embedded surfaces $\Sigma_i\subset X$ such that $[\Sigma_i]$ is Poincar\'{e} dual to $v_i$. We assume that the surfaces are generic in the sense that they are pairwise transverse, and $\Sigma_i\cap \Sigma_j \cap \Sigma_k=\emptyset$ for distinct $i,j,k$. Apply the construction of the previous paragraph to the collection $\mathcal{S}=\{\Sigma_1,\ldots,\Sigma_{n+1}\}$. In this way, to any subset $V\subset H^2(X)_\Z$ of size $n+1$ we construct a family of metrics parametrized by a permutahedron $P_n$, and obtain an associated period map
\begin{equation}\label{eq:periodmappermutahedron}
    \Pi_{X,V}: P_n \longrightarrow \overline{{\rm{Gr}}}^+(H^2(X)).
\end{equation}
The next proposition describes the behavior of this period map on the faces of the permutahedron $P_n$. The statement is more general than what is needed for Theorem \ref{thm:main2}. Write $V_I$ for the subspace of $H^2(X)$ spanned by $\{v_i\}_{i\in I}$. For a subspace $W\subset H^2(X)$, we write $W^\perp$ for its complement with respect to the intersection pairing, and $\text{null}(W)=W\cap W^\perp$.\\

\begin{prop}\label{prop:periodmapfaces}
Let $\Pi_{X,V}:P_n\to {\overline{\rm{Gr}}}^+(H^2(X))$ be the period map associated to a family of metrics parametrized by the permutathedron $P_n$, constructed from $V=\{v_1,\ldots,v_{n+1}\}\subset H^2(X)_\Z$ as above. For a face $F_{\mathbf{I}}\subset \partial P_n$ corresponding to $\mathbf{I}=\{I_i\}_{i=1}^l$ as in \eqref{eq:nestedseq}, and $g\in {\rm{int}}(F_\mathbf{I})$, we have
    \begin{equation}\label{eq:periodmapfaces}
    \Pi_{X,V}\left( g \right) = H_{l+1}^+ +  \sum_{i=1}^{l} H_i^+  + {\rm{null}}(V_{I_i}) 
    \end{equation}
    for some collection of $H_i^+ \in {\rm{Gr}}^+(V_{I_i} \cap V_{I_{i-1}}^{\perp})$, where $i=1,\ldots,l$, and $H_{l+1}^+\in 
 {\rm{Gr}}^+( V_{I_{l}}^{\perp}) $.
\end{prop}

\vspace{.2cm}

Note that in this statement, and in general, if $W$ is a vector space with a negative definite inner product, then $\text{Gr}^+(W)$ is a point, corresponding to the zero subspace of $W$. 

\vspace{.2cm}

\begin{proof}
    The hypersurfaces $Y_i:=Y_{I_i}$ decompose $X$ into the union of $X_i:= X_{I_i} \setminus \text{int}(X_{I_{i-1}})$ along their boundaries, where $i$ ranges from $1$ to $l+1$. Note $\partial X_i = Y_{i} \cup -Y_{i-1}$. By convention we define $Y_0=Y_{l+1}=\emptyset$ and also $X_{I_{l+1}}=X$. For $g\in {\rm{int}}(F_\mathbf{I})$ we have 
    \[
        \Pi_{X,V}\left( g \right) = \sum_{i=1}^{l+1} H_i^+ + \text{im}(\delta),
    \]
    where $H_i^+\subset H_c^2(X_i)$ maps to $\widehat{H}^2(X_i)$ as the space of $L^2$ self-dual harmonic $2$-forms on $X_i$, and 
    \[
        \delta:H^1(Y)\to H^2(X),
    \]
    where $Y=\cup_{i=1}^{l} Y_{i}$, and $\delta= \sum \delta_i$ where $\delta_i:H^1(Y_i)\to H^2(X)$ is from the Mayer--Vietoris sequence for $X$ decomposed along $Y_i$. We claim that the image of 
    \begin{equation}\label{eq:viimage}  
        H^2_c(X_{I_i})\to H^2(X)
    \end{equation}
    is equal to $V_i$, where $V_i:=V_{I_i}$ for $i=1,\ldots, l$ and $V_{l+1}=H^2(X)$. If $i=l+1$ this is by definition. Otherwise, $X_{I_i}$, being the regular neighborhood $\cup_{j\in I_i}\Sigma_j$, is a plumbing of disk bundles over surfaces, and $H_c^2(X_{I_i})$ is generated by the Poincar\'{e} duals of the surfaces $\Sigma_j$, which are sent under \eqref{eq:viimage} to the $v_j$, establishing the claim. Here it is important that the $\Sigma_i$ are connected, and also that the $\Sigma_i$ are pairwise transverse and $\Sigma_i\cap \Sigma_j \cap \Sigma_k=\emptyset$ for distinct $i,j,k$.\\
    
    Now $H^2_c(X_i)\to H^2(X)$ factors through \eqref{eq:viimage}, and is injective on $H_i^+$ (by the nondegeneracy of the intersection form on $H_i^+$). Thus $H_i^+\subset V_{i}$. As forms in $H_i^+$ are compactly supported outside of $X_{I_{i-1}}$ we have $H_i^+\subset V_{i-1}^\perp$. Further, $\text{im}(\delta_i)$ is in the nullspace of the pairing on $H^2_c(X_{I_i})$, and so maps into the nullspace of $V_i\subset H^2(X)$. We obtain $\text{im}(\delta)\subset \sum_{i=1}^l \text{null}(V_i)$.\\

    We have shown that the left side of \eqref{eq:periodmapfaces} is contained in the right side. To complete the proof it suffices to show that the dimension of the right hand side is equal to $b^+(X)$. Write $N_i=\text{null}(V_i)$. Then, in a similar way as is proved Lemma \ref{lemma:mv2}, we have for $1\leq i \leq l+1$:
    \[
        b^+(V_{i}) = b^+(V_{i-1}) + b^+(V_{i-1}^\perp \cap V_i) + |N_{i-1}| - |N_{i-1}\cap N_i|.
    \]
    Here $|W|$ is shorthand for the dimension of $W$, and $b^+(W)$ is the dimension of a maximal positive subspace of $W$. Iterating this identity, we obtain
    \begin{equation} \label{eq:indefidentitynullspaces}
        b^+(X) = b^+(V_{l+1}) = \sum_{i=1}^{l+1} b^+(V_{i-1}^\perp\cap V_i) + \sum_{i=1}^l|N_{i-1}| - |N_{i-1}\cap N_i |.
    \end{equation}
    Using $N_i=V_i\cap V_i^\perp$ and $V_{i+1}^\perp\subset V_i^{\perp}$, we verify $N_1\cap N_2 = N_1\cap \sum_{i=2}^l N_i$. Then
    \[
        |\sum_{i=1}^l N_i| =  |N_1| + |\sum_{i=2}^l N_i | - | N_1 \cap \sum_{i=2}^l N_i| = |N_1| -|N_1\cap N_2| + |\sum_{i=2}^l N_i|.
    \]
    Continuing in this fashion, $|\sum_{i=1}^l N_i|$ is equal to the last sum in \eqref{eq:indefidentitynullspaces}. The $H_i^+$ are pairwise orthogonal and $|H_i^+| = b^+(V_{i-1}^\perp\cap V_i)$, so \eqref{eq:indefidentitynullspaces} computes $b^+(X)=|\sum_{i=1}^{l+1} H_i^+  + N_i|$, as desired.
\end{proof}

\vspace{.2cm}

\begin{remark}\label{rem:11}{\emph{
Consider the conclusion of Proposition \ref{prop:periodmapfaces} when the intersection pairing on $V_i$ is non-degenerate for each $i$, with notation as above. Then we have a direct sum decomposition
\[
    H^2(X) = \bigoplus_{i=1}^{l+1} V_i\cap V_{i-1}^\perp,
\]
and \eqref{eq:periodmapfaces} says that an element in the face $F_\mathbf{I}$ associated to the nested sequence $\mathbf{I}=\{I_i\}_{i=1}^l$ is sent to a sum of positive subspaces with respect to this decomposition. Thus we have
\begin{equation}\label{eq:periodmappolytopenondeg}
    \Pi_{X,V}(F_\mathbf{I}) \subset \prod_{i=1}^{l+1} \text{Gr}^+( V_i\cap V_{i-1}^\perp )  \subset \text{Gr}^+(H^2(X)).
\end{equation}
If both $b^+(X)$ and $b^-(X)$ are greater than $1$, then the middle manifold in \eqref{eq:periodmappolytopenondeg} is of codimension at least $2$ in $\text{Gr}^+(H^2(X))$, whenever it is proper. A similar remark holds in the case in which one or more of the $V_i$ are degenerate. This is one reason that our proof of Theorem \ref{thm:main} given below does not directly generalize to the case when both $b^+(X)$ and $b^-(X)$ are greater than $1$.
}}
\end{remark}

\vspace{.2cm}

\subsection{The case $b^+=1$ and the proof of Theorem \ref{thm:main2}}\label{subsec:proofsurjectivity}

We now restrict to the case in which $b^+(X)=1$. Let $n=b_2(X)-1$. We may assume $n\geq 2$, the case $n=1$ having been explained at the introduction to this section. Choose an isometry between $H^2(X)$, with its intersection pairing, and $\R^{1,n}$, which is $\R^{n+1}$ equipped with the bilinear form
\begin{equation}\label{eq:standardbilinearform}
    x\cdot y = x_{0}y_0 - x_1y_1 -\cdots - x_n y_n.
\end{equation}
We also require that the integral lattice $H^2(X)_\Z$ is sent to $\Z^{1,n}\subset \R^{1,n}$ under this isomorphism. Here we use the fact that for $n\geq 2$, any unimodular integral symmetric bilinear form on $\Z^{n+1}$ of signature $(1,n)$ is isomorphic to $\Z^{1,n}$. (If $n=1$ there are two cases, as seen through Examples \ref{ex1} and \ref{ex2}.) Recall that $n$-dimensional hyperbolic space may be defined as the hyperboloid
\[
    \mathbb{H}^n = \{ x\in \R^{1,n} \; \mid \; x\cdot x = 1, \; x_0>0\} 
\]
with metric induced by the negative of \eqref{eq:standardbilinearform}. The space of positive lines $\ell\subset \R^{1,n}$ is identified with $\mathbb{H}^n$ by sending $\ell$ to $\ell\cap \mathbb{H}^n$. In this way we identify ${\rm{Gr}}^+(H^2(X))$ with $\mathbb{H}^n$, preserving rational points, i.e. points corresponding to lines spanned by rational vectors.\\

Let $V=\{v_1,\ldots, v_{n+1}\}\subset H^2(X)_\Z$ be a linearly independent set, and construct an associated family of metrics parametrized by $P_n$. We then have the period map
\[
    \Pi_{X,V}:P_n \longrightarrow \mathbb{H}^n
\]
and our goal is show that this is surjective. To achieve this we will choose $V$ such that it determines a simplex $\Delta_V\subset\mathbb{H}^n$, and show that $\Pi_{X,V}$ maps $P_n$ onto $\Delta_V$. To this end, we require an elementary lemma for maps of permutahedra to simplices.\\

\begin{figure}[t]
\centering
\hspace*{-0.25cm}\includegraphics[scale=0.5]{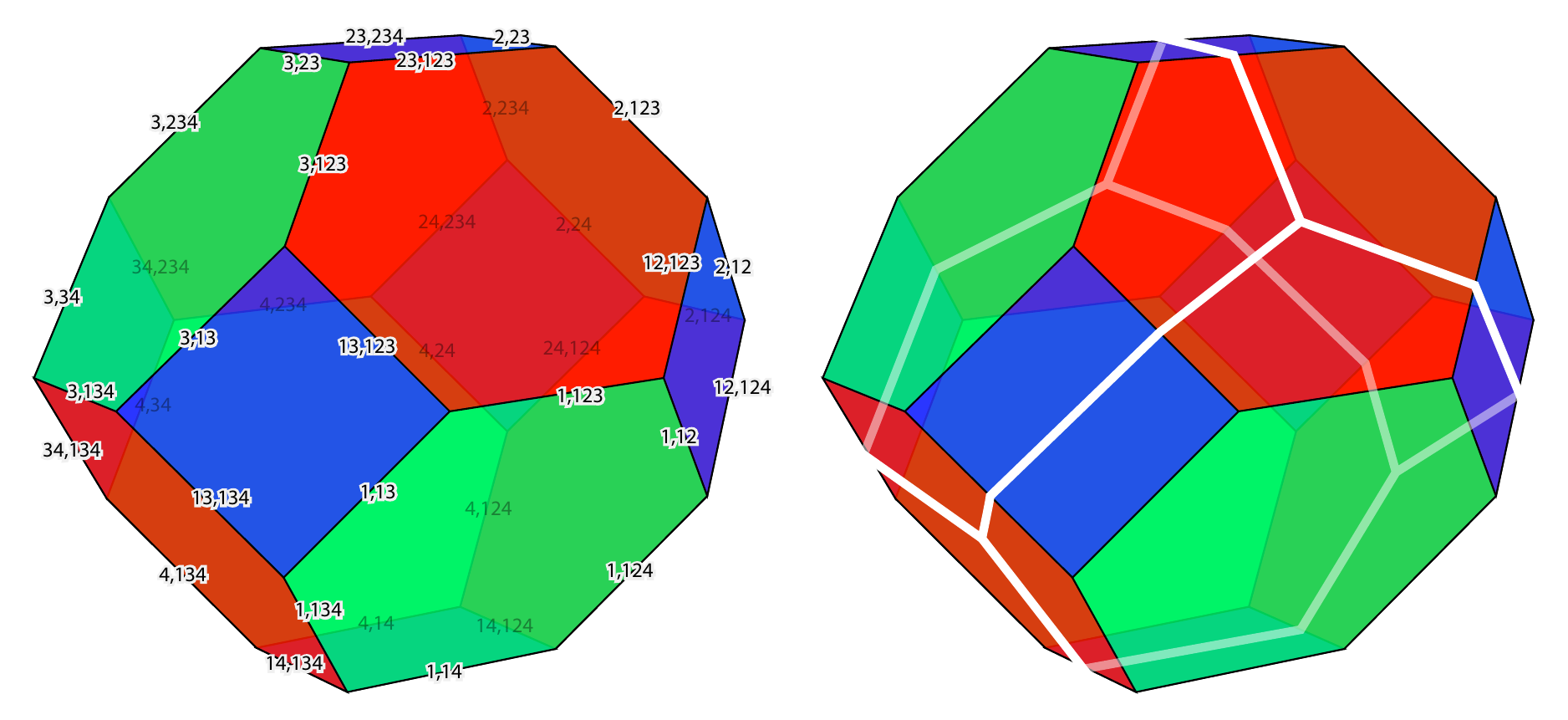}
\caption{\small{The permutahedron $P_3$. Left: Each edge is labelled; ``$34,134$'' is shorthand for the nested sequence $ \{3,4\} \subset \{1,3,4\}$. The green (resp. red) hexagonal faces correspond to nested sequences with one subset of size $1$ (resp. $3$), such as $\{1\}$ (resp. $\{1,2,3\}$). Square blue faces correspond to nested sequences with one subset of size $2$. Paint the faces of the $3$-simplex $\Delta_3$ green. Then $P_3$ is obtained from $\Delta_3$ by cutting off corners (revealing red), and further truncating (revealing blue). Right: The map $\mathbf{i}:\Delta_3\to P_3$ sends the $1$-skeleton of $\Delta_3$ into $\partial P_3$ as shown. The map $\mathfrak{F}:P_3\to \Delta_3$ is the result of collapsing the truncated faces (red collapses to vertices, and blue collapses to edges).}}\label{fig:p3}
\end{figure}

The $n$-simplex $\Delta_n$ may be described combinatorially as having its proper faces labelled by nonempty proper subsets $I\subset \{1,\ldots,n+1\}$. The face $F_I$ is of codimension $|I|$, and $F_I\subset F_J$ if and only if $J\subset I$. This is the description of $\Delta_n$ that arises from realizing it as the region in $n$-dimensional space bounded by a collection of hyperplanes labelled by $1,\ldots,n+1$.\\ 

There is a ``forgetful'' map $\mathfrak{F}:P_{n}\to \Delta_n$ which on the boundary maps the face $F_\mathbf{I}$ to $F_{I_l}$ where $I_l$ is the maximal proper subset appearing in the nested sequence $\mathbf{I}$. In fact, the permutahedron $P_{n}$ can be viewed as a truncation of the simplex $\Delta_n$, see \cite{bloom}, and also Figures \ref{fig:hexagon} and \ref{fig:p3}. From this viewpoint, the map $\mathfrak{F}$ is the result of collapsing faces on $P_n$ introduced by truncation onto corresponding faces in $\Delta_n$. We have another map
\[
    \mathbf{i}:\Delta_n\longrightarrow P_n
\]
which is described as follows: viewing $P_{n}$ inside of $\Delta_n$ as the result of truncation, send a given point in $\Delta_n$ to its closest point in $P_{n}$. For any face $F_I\subset \partial \Delta_n$, we have
\[
    \mathbf{i}(F_I)\subset \bigcup F_\mathbf{I}\subset \partial P_n
\]
where the union is over all $\mathbf{I}=\{I_i\}_{i=1}^l$ satisfying $I\subset I_l$. Further, $\mathfrak{F}\circ \mathfrak{i}=\text{id}_{\partial \Delta_n}$. The following says that if a map sends faces to faces in the same manner as does $\mathfrak{F}$, then it is surjective.

\vspace{.2cm}

\begin{lemma}\label{lemma:surjectivemapofpolytopes}
    Let $f:P_{n}\to \Delta_n$ be a continuous map such that for each face $F_\mathbf{I}\subset \partial P_{n}$ corresponding to a nested sequence $\mathbf{I}=\{I_i\}_{i=1}^l$ we have  $f(F_\mathbf{I})\subset F_{I_l}\subset \partial \Delta_n$. Then $f$ is surjective.
\end{lemma}

\vspace{.15cm}

\begin{proof}
    Let $F_I$ be a face of $\Delta_n$ corresponding to $I\subset \{1,\ldots,n+1\}$. Then
    \[
        f(\mathbf{i}(F_I)) \subset \bigcup  f\left( F_\mathbf{I} \right) 
    \]
    where the union is over all nested sequences $\mathbf{I}$ with $I\subset I_l$. By assumption, each $f(F_\mathbf{I})\subset F_{I_l}$. On the other hand, $F_{I_l}\subset F_I$ since $I\subset I_l$. Thus $f(\mathbf{i}(F_I)) \subset F_{I}$. Now $f\circ \mathbf{i}$ is a self-map on the simplex $\Delta_n$ which maps each face into itself. It is straightforward to prove that any such map is homotopic rel $\partial \Delta_n$ to the identity map. Then $f\circ \mathbf{i}$ is surjective, as it has degree $1$ as a map on $(\Delta_n,\partial \Delta_n)$. The surjectivity of $f$ follows.
\end{proof}

\vspace{.2cm}

With this background, we now prove Theorem \ref{thm:main2}.

\begin{proof}[Proof of Theorem \ref{thm:main2}]
Let $V=\{v_1,\ldots, v_{n+1}\}\subset H^2(X)_\Z$ be a linearly independent set such that $V_I$ is negative definite for all subsets $I\subset \{1,\ldots,n+1\}$ of cardinality $n$. The hyperplanes 
\[
   H_i:= v_i^\perp \cap \mathbb{H}^n 
\]
bound an $n$-simplex $\Delta_V\subset \mathbb{H}^n$, and every compact hyperbolic $n$-simplex in $\mathbb{H}^n$ with rationally defined vertices arises in this fashion. We remark that every point in $\mathbb{H}^n$ is in the interior of such a simplex. Thus to show that $\Pi_X$ is onto, it suffices to show that
\begin{equation}\label{eq:proofinteriors}
    \text{int}(\Delta_V) \subset \Pi_{X,V}(\text{int}{(P_{n})})
\end{equation}
where $\Pi_{X,V}$ is the period map restricted to the $n$-dimensional permutahedron family of metrics associated to $V$. Here it is important that $\Pi_{X,V}(\text{int}{(P_{n})})$ is in the image of the period map, as the interior of $P_{n}$ parametrizes smooth (unbroken) metrics.\\

We first determine the behavior of $\Pi_{X,V}$ on the boundary of $P_{n}$. Let $\mathbf{I}=\{I_i\}_{i=1}^l$ be a nested sequence of proper non-empty subsets of $\{1,\ldots,n+1\}$, and $F_\mathbf{I}$ the corresponding codimension $l$ face of $P_{n}$. Then an application of Proposition \ref{prop:periodmapfaces} yields
\[
    \Pi_{X,V}(F_\mathbf{I}) \subset \bigcap_{i\in I_l} H_i.
\]
Next, choose a retraction $r:\mathbb{H}^n\to \Delta_V$ with the property that $r(H_i)\subset \Delta_V\cap H_i$. For example, $r$ can be the map which sends a point in $\mathbb{H}^n$ to its closest point in $\Delta_V$. Then the map
\[
    r\circ \Pi_{X,V}:P_n\longrightarrow  \Delta_V
\]
sends $F_\mathbf{I}$ into $\cap_{i\in I_l} H_i\cap \Delta_V$, the face of $\Delta_V$ corresponding to $I_l$. By Lemma \ref{lemma:surjectivemapofpolytopes}, we obtain that $r\circ \Pi_{X,V}$ surjects onto $\Delta_V$. As $r$ is the identity on $\Delta_V$, we obtain  
\begin{equation}\label{eq:proofinteriors2}
    \Delta_V \subset \Pi_{X,V}(P_{n}).
\end{equation}
Finally, $\Pi_{X,V}$ maps $\partial P_n$ into $\partial \Delta_V$, so we obtain \eqref{eq:proofinteriors}, as desired.
\end{proof}

\vspace{.35cm}

\begin{remark}\label{rem:12}{\emph{
Following Remark \ref{rem:1}, the proof of Theorem \ref{thm:main2} given above adapts to the case in which $\dim X = 4k$ for any positive integer $k$, with no essential changes.
}}
\end{remark}

\vspace{.35cm}

Let $V=\{v_1,\ldots,v_{k+1}\}\subset H^2(X)_\Z$ be a linearly independent set, and $\Pi_{X,V}$ the associated period map. In the above proof, Proposition \ref{prop:periodmapfaces} was used in the simplified case in which $V$ determines a bounded simplex in $\mathbb{H}^n$. More generally, for a nested sequence $\mathbf{I}=\{I_i\}_{i=1}^l$ set
\[
    i^+ := \min\{ j \; \mid \; V_{I_j} \text{ not negative definite }\}.
\]
Let $g\in F_{\mathbf{I}}$. If $V_{I^+}$ is degenerate, it has a 1-d null space, and $\Pi_{X,V}(g)=\text{null}(V_{I^+})$. Otherwise, 
\[
    \Pi_{X,V}(g) \subset V_{I_{i^+}}\cap V_{I_{i^+-1}}^\perp.
\]
In many cases, these conditions cut out a region in $\mathbb{H}^n$ which contains an $n$-dimensional polytope, possibly with ideal points on the sphere at infinity, and the above proof carries over to show that $\Pi_{X,V}(P_n)$ contains this polyhedron. In the next subsection we see this in some examples.\\

We remark on the role of permutahedra in the above proof. These polytopes are universal from our viewpoint, in that they parametrize a family of metrics for any collection of surfaces in $X$, regardless of the intersection pairings of the surfaces. On the other hand, if certain subcollections of surfaces do not intersect, one can often work with simpler polytopes. For example, if $k+1$ surfaces are pairwise disjoint, one can construct an associated metric family which is a $k$-simplex.\\

Along the same lines, one might ask to prove Theorem \ref{thm:main2} using pairwise disjoint surfaces with negative self-intersection, with polytopes whose codimension $1$ faces are labelled by these surfaces. One is led to search for what are called right-angled (finite volume) hyperbolic polyhedra. However, such polyhedra do not exist in $\mathbb{H}^n$ for $n>12$ \cite{pv, dufour}. This illustrates the utility of working with the general construction, avoiding conditions on how the surfaces intersect.

\vspace{.2cm}

\begin{remark}\label{rem:3}{\emph{
Following Remark \ref{rem:2}, the above proof also shows that for connected $X$ with boundary, the period map $\text{Met}_{\text{cyl}}(X)\to \text{Gr}^+(\widehat{H}^2(X))$ is surjective when $b^+(X)=1$. In this case, in the isomorphism $\widehat{H}^2(X)\to\R^{1,n}$ it may not be possible to arrange that $\widehat{H}^2(X)_\Z$ maps onto $\Z^{1,n}$. This is not important for the argument, however -- only that $\R\cdot\widehat{H}^2(X)_\Z$ is dense in the positive Grassmannian (so that there are enough simplices in $\mathbb{H}^n$ defined over these points).
}}
\end{remark}

\vspace{.3cm}

\subsection{Examples in the hyperbolic plane}\label{subsec:examples}

To illustrate the construction used in the proof of Theorem \ref{thm:main2}, we consider examples in the case that $b^+(X)=1$ and $b_2(X)=3$. Upon choosing an isometry of $H^2(X)$ with $\R^{1,2}$, the positive Grassmannian is identified with the hyperbolic plane $\mathbb{H}^2$. As before, we make these choices so that the integral lattice $H^2(X)_\Z$ is sent to $\Z^{1,2}\subset \R^{1,2}$.\\

Let $V=\{v_1,v_2,v_3\}\subset H^2(X)_\Z$, which we also view as a subset of $\Z^{1,2}\subset \R^{1,2}$. For simplicity we assume that $V$ is a linearly independent set. By choosing connected embedded surfaces $\Sigma_1,\Sigma_2,\Sigma_3\subset X$ which are Poincar\'{e} dual to $v_1,v_2,v_3$, we construct a family of (broken) metrics on $X$ parametrized by the permutahedron $P_2$. As shown in Figure \ref{fig:hexagon}, $P_2$ is a hexagon. We have an associated period map $\Pi_{X,V}:P_2 \to \mathbb{H}^2$ for this family of metrics.\\ 

The codimension 1 faces of $P_2$ are $F_{\{i\}}$ and $F_{\{i,j\}}$ where $i,j\in \{1,2,3\}$ and $i\neq j$. Proposition \ref{prop:periodmapfaces} determines the behavior of the period map on these faces. See Table \ref{table}. Briefly, the image of face $F_{\{i\}}$ (resp. $F_{\{i,j\}}$) under $\Pi_{X,V}$ is constrained in three possible ways, depending on the isomorphism type of the bilinear form restricted to the subspace $\langle v_i\rangle$ (resp. $\langle v_i, v_j\rangle$).\\

\begin{figure}[t]
\centering
\def\arraystretch{1.75}%
\begin{tabular}{l|l|l}
 Condition & Constraint on period map & Type   \\
 \hline
$v_i\cdot v_i < 0$ & $\Pi_{X,V}(F_{\{i\}}) \subset \langle v_i\rangle^\perp \cap \mathbb{H}^2$  & geodesic \\
 $v_i\cdot v_i =  0$ &   $\Pi_{X,V}(F_{\{i\}}) = \langle v_i\rangle  \cap \partial \mathbb{H}^2$ & ideal point \\
 $v_i\cdot v_i > 0$&    $\Pi_{X,V}(F_{\{i\}}) =  \langle v_i\rangle  \cap  \mathbb{H}^2$ & point   \\
 $\forall v\in \langle v_i,v_j\rangle\, \setminus \,\{0\}$\, :\, $v\cdot v<0$ \qquad&   $\Pi_{X,V}(F_{\{i,j\}}) \subset  \langle v_i,v_j\rangle^\perp  \cap  \mathbb{H}^2$ \qquad & geodesic \\
   $\exists v \in \langle v_i,v_j\rangle\, \setminus \,\{0\}$\,:\, $v\cdot v=0$&  $\Pi_{X,V}(F_{\{i,j\}}) = \langle v \rangle  \cap  \partial \mathbb{H}^2$ & ideal point \\
  $ \exists v \in \langle v_i,v_j\rangle$  \,:\,  $v\cdot v>0$ &  $\Pi_{X,V}(F_{\{i,j\}}) = \langle v \rangle  \cap  \mathbb{H}^2$ & point \\
\end{tabular}
\captionof{table}[table]{{\small{Each condition in the left column gives a constraint on where the period map sends a face of the hexagon $P_2$. The types of subsets of $\mathbb{H}^2$ appearing are described in the right column. The notation $\langle \cdot \rangle$ denotes span of vectors, and $\partial \mathbb{H}^2$ denotes the ideal boundary of $\mathbb{H}^2$, which is a circle.}}}\label{table}
\end{figure}

A particularly symmetric family of examples $V=\{v_1,v_2,v_3\}$ is given by
\begin{equation}\label{eq:symmetricfamily}
    v_{i+1} = \left( 1 \; | \; a\cos(2 i \pi/3), \; a\sin(2 i \pi/3) \right), \qquad i\in \{0,1,2\}.
\end{equation}
where $a\in \R\setminus \{0\}$. Here and below we write vectors in $\R^{1,2}$ as $x=(x_0 \, | \, x_1, x_2)$, to emphasize the signature of the bilinear form. Each $V$ in this family is not integral (nor rational), but nonetheless illustrates the behavior of the face constraints of the period map. When $a>2$, the $v_i^\perp$ give geodesics forming a hyperbolic triangle. However, for all $a\neq 0$, the constraints on the period map give some finite area polygon in $\mathbb{H}^2$. See Figure \ref{fig:examples}, where illustations are given in the Poincar\'{e} disk for $a>0$. When $a<0$ there are similar pictures, with some colors interchanged.\\

\begin{figure}[p!]
\centering
\hspace*{-0.25cm}\includegraphics[scale=1.05]{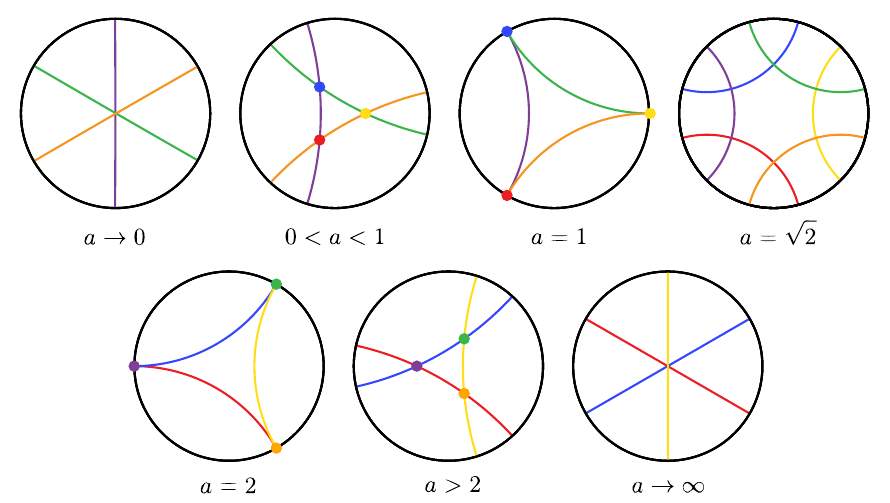}
\caption{\small{The subsets defining the face constraints from Table \ref{table}, given for the family of $V=\{v_1,v_2,v_3\}$ defined in \eqref{eq:symmetricfamily}. The colors correspond to the face colors of the hexagon in Figure \ref{fig:hexagon}.}}\label{fig:examples}
\end{figure}

\begin{figure}[p!]
\centering
\vspace{.25cm}
\hspace*{-0.25cm}\includegraphics[scale=1.05]{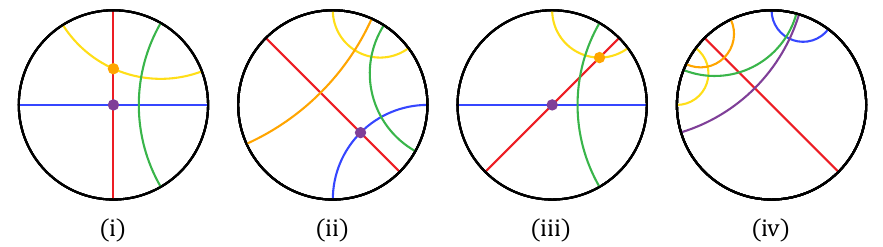}
\caption{\small{Examples of face constraint configurations for some other choices of $V$.}}\label{fig:examples2}
\end{figure}

There are many other types of configurations of geodesics and points that arise from the face constraints given in Table \ref{table}. For example, here are some sets $V=\{v_1,v_2,v_3\}\subset H^2(X)_\Z = \Z^{1,2}$, which give rise to the configurations depicted in Figure \ref{fig:examples2}:
\begin{center}
\def\arraystretch{1.65}%
\begin{tabular}{llll}
\text{(i):} & $v_1 = (0 \,| \, 1,0 )$ & $v_2 = (0\,| \,0,1)$ & $v_3 = (2\,| \,1,3)$\\
\text{(ii):} & $v_1 = (0\,| \,1,1)$ & $v_2 = (1\,| \,1,-1)$ & $v_3 = (2\,| \,1,2)$\\
\text{(iii):}&  $v_1 = (0\,| \,1,-1)$ & $v_2 = (0\,| \,0,1)$ & $v_3 = (2\,| \,1,2)$\\
\text{(iv):} & $v_1 = (0\,| \,1,1)$ & $v_2 = (3\,| \,1,3)$ & $v_3 = (3\,| \,-3,1)$\\
\end{tabular}
\end{center}

\vspace{.25cm}

Not every linearly independent $V$ gives a configuration which encloses a polygon with non-empty interior (and thus the image of the corresponding period map on $P_2$ is not guaranteed to have non-empty interior). For a simple example, take $v_1 = (1 \, | \,1,1)$, $v_2 = (0\, | \,1,1)$, $v_3 = (0 \, | \,1,-1)$.\\

\appendix

\section{Gluing harmonic forms}\label{appendix}
Here we give a proof of Proposition \ref{prop:main}, by showing how it follows from well-known gluing results for harmonic forms. We will follow \cite{clmi}. For similar gluing constructions, see also \cite{donaldson-book}. We assume the notation set up in the paragraph of \eqref{eq:maindecomp}. Write $\| \cdot \|_Y$ for the $L^2$ metric on $\Omega^\ast(Y)$ induced by $g_Y$, and so forth. In this notation we assume the metric on any cylinder is the product metric: $\| \cdot \|_{X}$ is defined using the metric $h(1)$ on $X=X(1)$, and $\|\cdot \|_{X(r)}$ using $h(r)$ on $X(r)$.\\

\noindent {\textbf{The gluing map}}

\vspace{.1cm}

\noindent We recall some results from \cite{clmi}. First, a splicing map is defined:
\begin{equation}\label{eq:slicingmap}
    \Phi_r:\mathcal{H}^+_{X_1}\oplus \mathcal{V}_Y \oplus \mathcal{H}^+_{X_2}\to \Omega^+_{h(r)}(X(r))
\end{equation}
Recall $\mathcal{H}^+_{X_i}$ is the space of $L^2$ harmonic self-dual $2$-forms on $X_i(\infty)$ with its cylindrical end metric. 
We call a 2-form $\omega_1$ on $X_1(\infty)$ an {\emph{extended}} $L^2$ harmonic form if it is harmonic and
\[
    \|\omega_1 \|_{X_1} <\infty, \qquad \|\omega_1- \pi^\ast (\alpha \wedge dt + \beta) \|_{Y\times (0,\infty) } < \infty
\]
for some harmonic forms $\alpha\in \mathcal{H}_Y^1$, $\beta\in \mathcal{H}^2_Y$, and we say that $\omega_1$ {\emph{extends}} $\alpha \wedge dt + \beta$. We have written $\pi:Y\times (0,\infty)\to  Y$ for projection, and $\pi^\ast$ is extended to $\Omega^\ast(Y)\wedge dt$ by $\pi^\ast(\alpha\wedge dt) = \pi^\ast(\alpha)\wedge dt$. A similar notion of extended $L^2$ harmonic forms is defined for $X_2(\infty)$. The space $\mathcal{V}_Y$ is the vector space of harmonic $1$-forms $\alpha$ on $Y$ such that there exist extended $L^2$ harmonic self-dual $2$-forms $\omega_i$ on $X_i(\infty)$ extending $\alpha\wedge dt + \star_Y \alpha$. The map $\delta:H^1(Y)\to H^2(X)$ induces
\[
    \mathcal{V}_Y \cong \text{im}(\delta),
\]
and also $ \mathcal{V}_Y\cong \text{im}(\iota^\ast: H^2(X)\to H^2(Y) )$ by the map $\alpha\mapsto [\star_Y \alpha]$. These last claims follow from Lemma B.1 of \cite{clmi}, restricted to the case of self-dual harmonic $2$-forms on a $4$-manifold.\\

Let $\rho:\R\to [0,1]$ be a smooth non-decreasing function satisfying $|\rho'|\leq 4$ and
\[
    \rho(t) = \begin{cases} 0, &\;\; t\leq 1/4\\
                            1, &\;\; t\geq 3/4
                            \end{cases}
\]
To describe \eqref{eq:slicingmap}, we first declare that for any element $\omega$ in its domain, $\Phi_r(\omega)|_{X_i}=\omega$, for $i=1,2$. Now consider $\omega_1\in \mathcal{H}^+_{X_1}$. We define $\Phi_r(\omega_1)$ elsewhere by:
\begin{gather*}
  \Phi_r(\omega_1)|_{Y\times [-r,0]}(y,t) = \rho(-t) \omega_1(y,t+r), \qquad \Phi_r(\omega_1)|_{X_2(r)}=0
\end{gather*}
The translation appears by the way in which we identify $X_1(r)=X_1\cup Y\times [0,r]$ as a subset of $X(r)$. The description of $\Phi_r$ for elements of $\mathcal{H}^+_{X_2}$ is similar.\\

Next, let $\beta\in \mathcal{V}_Y$. There exist unique extended $L^2$ self-dual harmonic $2$-forms $\omega_i$ on $X_i(\infty)$, extending $\omega_\alpha:=\alpha \wedge dt +  \star_Y \alpha$, such that $\omega_i$ is $L^2$ orthogonal to $\mathcal{H}^+_{X_i}$. Then
\begin{align*}
 \Phi_r(\beta)|_{Y\times [-r,0]}(y,t) &= \rho(-t)(\omega_1(y,t+r) - \pi^\ast\omega_\alpha(y) ) + \pi^\ast\omega_\alpha(y) \\[3mm]
\Phi_r(\beta)|_{Y\times [0,+r]}(y,t) &= \rho(+t)(\omega_2(y,t-r) - \pi^\ast\omega_\alpha(y) ) + \pi^\ast\omega_\alpha(y)
\end{align*}

\vspace{.1cm}

Let $\mathbf{P}_r$ denote $L^2$ projection of $2$-forms on $X(r)$ to $\mathcal{H}^+_{X(r)}$, the space of $h(r)$-harmonic $2$-forms. Note that the image of $\mathbf{P}_r \Phi_r$ lies in the $h(r)$-self-dual harmonic $2$-forms. It follows from the results in \cite{clmi} (see Lemma 4.1 and Appendix B) that there exists an $R$ such that for $r\geq R$, we have 
\begin{equation}\label{eq:expdecay}
    \| \mathbf{P}_r \Phi_r(\omega) -\Phi_r(\omega) \|_{X(r)} \leq e^{-cr} \|\Phi_r(\omega)\|_{X(r)}
\end{equation}
where $c>0$ is a constant, independent of $r$, determined by the spectrum of the $g_Y$-Laplacian on forms of $Y$. Furthermore, for large $r$, the following composition is an isomorphism:
\begin{equation*}\label{eq:slicingmapcomposed}
    \mathbf{P}_r\circ\Phi_r:\mathcal{H}^+_{X_1}\oplus \mathcal{V}_Y \oplus \mathcal{H}^+_{X_2}\to \mathcal{H}^+_{X(r)}
\end{equation*}

\vspace{.15cm}

\noindent {\textbf{Comparison of metrics}}

\vspace{.1cm}

\noindent Let us be more explicit about the diffeomorphisms $f_r:X\to X(r)$. Always assume $r\geq 1$. Write $s$ for the coordinate in $[-1,1]$, and $t$ in $[-r,r]$. Choose diffeomorphisms $\chi_r:[-1,1]\to [-r,r]$ such that $\chi_r'(s)=1$ for $1-|s|<\varepsilon$ and for all $s\in[-1,1]$, $\chi_r'(s)\geq 1$ and $\chi'_r(s)\leq Cr$ for some fixed $C>1$, independent of $r$. Then define diffeomorphisms
\begin{gather*}
    f_r:X\to X(r),\\
    f_r|_{X_i} = \text{id}_{X_i}, \qquad 
    f_r|_{Y\times [-1,1]} = \text{id}_{Y} \times \chi_r.
\end{gather*}

Let $\omega\in \Omega^2(X(r))$. Write $\omega_i=\omega|_{X_i}$. On $Y\times [-r,r]$ we can write $\omega = \alpha_t \wedge dt + \beta_t$ where $\alpha_t\in \Omega^1(Y)$ and $\beta_t\in \Omega^2(Y)$ depend smoothly on $t$. Then
\[
    \| \omega\|_{X(r)}^2 = \sum_{i=1}^2\| \omega\|_{X_i}^2+ \int_{-1}^1 \chi_r'(s) (\| \alpha_s \|_Y^2 + \|\beta_s\|_Y^2 ) ds
\]
where $s=\chi_r^{-1}(t)$. Turning next to $f_r^\ast \omega$, we have
\[
    \| f_r^\ast \omega\|_{X}^2 = \sum_{i=1}^2\| \omega\|_{X_i}^2+ \int_{-1}^1 (\chi_r'(s)^2\| \alpha_s \|_Y^2 + \|\beta_s\|_Y^2 ) ds,
\]
which is the squared $L^2$-norm of $\omega$ defined using the metric $g(r)$. We see that
\begin{equation}\label{eq:metriccomparison}
    \frac{1}{Cr} \| \omega\|^2_{X(r)} \leq \| f_r^\ast \omega\|^2_{X} \leq Cr \| \omega\|^2_{X(r)}.
\end{equation}

\vspace{.3cm}

\noindent {\textbf{Proof of Proposition \ref{prop:main}}}

\vspace{.1cm}

\noindent With this background established, we now turn to prove Proposition \ref{prop:main}. Note that $f_r^\ast \mathbf{P}_r$ is equal to the projection onto the $g(r)$-self-dual 2-forms of $X$. Thus our goal is to show that the image of $f_r^\ast \mathbf{P}_r\Phi_r$ as $r\to \infty$ converges to the subspace in \eqref{eq:propmain}.\\

We begin with some elementary observations. If $\{\omega_i\}_{i=1}^{\infty}$ and $\omega$ are closed forms in $\Omega^k(X)$, then to show $[\omega_i]\to [\omega]$ in $H^k(X)$, it suffices to show that 
\begin{equation}\label{eq:weakconvergence}
    \lim_{i\to \infty}\int_X \omega_i\wedge \eta = \int_X \omega\wedge \eta
\end{equation}
for every closed form $\eta$ of complementary degree. Further, suppose $X$ has a metric inducing an $L^2$ norm on forms, and suppose $\|\omega_i-\omega_i'\|_{L^2}\to 0$ for some other sequence of (not necessarily closed) forms $\{\omega'_i\}_{i=1}^{\infty}$. Then to show $[\omega_i]\to [\omega]$, it suffices to show \eqref{eq:weakconvergence} with $\omega'_i$ replacing $\omega_i$.\\

\noindent \textbf{Step 1\;\;} Let $\alpha\in \mathcal{V}_Y$. Then we claim that the following holds in $H^2(X)$:
\begin{equation}\label{eq:secondstep}
    \lim_{r\to \infty} [\frac{1}{2r}  f_r^\ast \mathbf{P}_r\Phi_r(\alpha)]  = \delta[\alpha].
\end{equation}
We first consider the forms $\frac{1}{2r}  f_r^\ast \Phi_r(\alpha)$ without the projection $\mathbf{P}_r$. We have
\begin{equation}\label{eq:1over2rexp}
    \frac{1}{2r}f^\ast_r\Phi_r(\alpha) = \frac{1}{2r} \chi_{Y\times (-1,1)} f_r^\ast\omega'+ \frac{1}{2r} \chi_{Y\times (-1,1)} f_r^\ast  \omega_\alpha + \frac{1}{2r}\sum_{i=1,2}\chi_{X_i}\omega_i
\end{equation}
where $\omega'=\Phi_r(\beta)-\pi^\ast(\omega_\alpha)$, recalling $\omega_\alpha=\alpha \wedge dt +  \star_Y \alpha$. The $\chi_S$ are characteristic functions. The terms $\chi_{X_i}\omega_i$ have $L^2$ norm on $X$ independent of $r$, and so the last term in \eqref{eq:1over2rexp} converges in $L^2$ to zero. Utilizing \eqref{eq:metriccomparison}, the $L^2$ norm of the first term in \eqref{eq:1over2rexp} satisfies
\[
     \| \frac{1}{2r} f_r^\ast \omega' \|_{L^2(Y\times (-1,1))}\leq \left(\frac{C}{r}\right)^{1/2} \left( \| \omega_1 - \pi^\ast \omega_\alpha \|_{L^2(Y\times (0,\infty))} + \| \omega_2 - \pi^\ast \omega_\alpha \|_{L^2(Y\times (-\infty,0))} \right)
\]
Here $\omega_i$ are extended $L^2$ self-dual harmonic $2$-forms on $X_i(\infty)$ extending $\omega_\alpha$. Thus the first term in \eqref{eq:1over2rexp} converges in $L^2$ to zero as $r\to \infty$. The middle term in \eqref{eq:1over2rexp} is
\[
\omega_r := \frac{1}{2r} \chi_{Y\times (-1,1)} f_r^\ast  \omega_\alpha = \frac{1}{2r} \chi_{Y\times (-1,1)}\left(\star_Y \alpha + \chi_r'(s) \alpha \wedge ds\right).
\]
At this point we may ask that the functions $\chi_r$ are chosen such that $\chi'_r/2r$ converge in $L^2$ to a smooth bump function of integral $1$, matching the description \eqref{eq:deltamap}. On the other hand, all we need is the weaker convergence \eqref{eq:weakconvergence}. To this end, let $\eta\in \Omega^2(X)$ be closed. On $Y\times (-1,1)$, we may write $\eta = \alpha_1\wedge ds + \alpha_2$, where $\alpha_i=\alpha_i(s)$ are $i$-forms on $Y$ depending smoothly on $s\in (-1,1)$. Then
\begin{align*}
    \int_X \omega_r   \wedge \eta  = \frac{1}{2r}
    \int_{Y\times (-1,1)}  \star_Y \alpha \wedge\alpha_1 \wedge ds + \frac{1}{2r}\int_{Y\times (-1,1)} \chi_r'(s)\alpha_2 \wedge  \alpha \wedge ds  
\end{align*}
The integral appearing in the first term on the right side is independent of $r$, and thus the first term converges to zero as $r\to \infty$. The second term is
\[
    \frac{1}{2r}\int_{-1}^1 \chi_r'(s)\left(\int_{Y} \alpha_2 \wedge \alpha \right)ds = \int_Y \alpha_2 \wedge \alpha = \int_{Y} \iota^\ast(\eta)\wedge \alpha,
\]
where we have used that the integral over $Y=Y\times \{s\}$ appearing does not depend on the parameter $s$. This follows from $\eta$ being closed, which implies $[\alpha_2(s)]\in H^2(Y)$ is independent of $s$. We obtain
\[
    \lim_{r\to \infty } \int_X\omega_r \wedge \eta =\int_X \eta \wedge \delta(\alpha),
\]
as desired. Thus far, our argument implies that for all closed $\eta\in \Omega^2(X)$, we have
\begin{equation}\label{eq:secondstep2}
    \lim_{r\to \infty } \frac{1}{2r}f^\ast_r\Phi_r(\alpha) \wedge \eta = \int_X  \delta(\alpha) \wedge \eta.
\end{equation}

We next observe the following, using inequalities \eqref{eq:expdecay} and \eqref{eq:metriccomparison}:
\begin{equation}\label{eq:expdecayagain}
    \| \frac{1}{2r} f_r^\ast  {\mathbf{P}_r}\Phi_r(\alpha) -  \frac{1}{2r} f_r^\ast \Phi_r(\alpha) \|_{X} \leq Cr e^{-cr}\| \frac{1}{2r}f_r^\ast\Phi_r(\alpha)\|_{X}.  
\end{equation}
From the previous paragraph, we have that $\|\frac{1}{2r}f_r^\ast\Phi_r(\alpha)-\omega_r\|_{X}$ converges to zero as $r\to \infty$. Furthermore, by direct calculation, we have
\[
    \| \omega_r \|^2_X = \frac{1}{2r^2}\int_{-1}^1 \chi_r'(s)\|\alpha\|_Y^2 \leq \frac{C}{r} \|\alpha\|_Y^2.
\]
It follows from this and \eqref{eq:expdecayagain} that $\frac{1}{2r} f_r^\ast  {\mathbf{P}_r}\Phi_r(\alpha) -  \frac{1}{2r} f_r^\ast \Phi_r(\alpha)$ converges to zero in $L^2$. With \eqref{eq:secondstep2}, this completes the proof of claim \eqref{eq:secondstep}.\\

\vspace{0.1cm}

\noindent \textbf{Step 2\;\;} Let $\omega\in \mathcal{H}^+_{X_1}$. Then $[f_r^\ast \mathbf{P}_r\Phi_r(\omega)]\in H^2(X)$. To make sense of the de Rham class of $\omega$ on $X$, we recall from \cite{aps-i} that on $X_1(\infty)$, $\omega$ is cohomologous (in the sense of currents) to a smooth form $\omega_c$ with support on $X_1\subset X=X(1)$, say. Then $[\omega_c]\in H^2(X)$ makes sense. A different choice of $\omega_c$ will give the same cohomology class in $H^2(X)$ up to $\text{im}(\delta)$. We claim that
\begin{equation}\label{eq:firststep}
    \lim_{r\to \infty}[f_r^\ast \mathbf{P}_r\Phi_r(\omega)] = [\omega_c]\;\;\text{ mod }\text{im}(\delta)
\end{equation}
in $H^2(X)/\text{im}(\delta)$. A similar claim holds for $X_2$ replacing $X_1$.\\ 

As in Step 1, we use inequalities \eqref{eq:expdecay} and \eqref{eq:metriccomparison} to observe
\begin{equation}\label{eq:expdecayagain2}
    \|  f_r^\ast  {\mathbf{P}_r}\Phi_r(\omega) - f_r^\ast \Phi_r(\omega) \|_{X} \leq Cr e^{-cr}\| \Phi_r(\omega)\|_{X}.  
\end{equation}
Furthermore, $\| \Phi_r(\omega)\|_{X} \leq \| \omega\|_{X_1(\infty)} $ by the construction of $\Phi_r(\omega)$, which applies a cutoff function to $\omega$. Thus $\|f_r^\ast  \mathbf{P}_r\Phi_r(\omega)  - f_r^\ast \Phi_r(\omega)\|_{X}\to 0$ as $r \to \infty$.\\

Next, suppose we have closed $2$-forms $\eta_j$ such that $[\eta_j]$ gives a basis for $H^2(X)$. Then, to show $\lim_{r\to \infty}[f_r^\ast \mathbf{P}_r\Phi_r(\omega)] = [\omega_c]$ in $H^2(X)$ from here, it would suffice to show for all $j$:
\begin{equation}\label{eq:weak2}
    \lim_{r\to \infty} \int_X f_r^\ast \Phi_r(\omega) \wedge \eta_j = \int_X \omega_c \wedge \eta_j.
\end{equation}
Choose compactly supported closed 2-forms $\eta_{i,j}$ on $X_i$ such that $[\eta_{i,j}]$ give a basis of $H_i\subset H_c^2(X_i)$, a subspace which maps isomorphically to $\widehat{H}(X_i)$. Note each $\eta_{i,j}$ may be viewed as a form on $X$. We may further choose closed forms $\eta_k'$ with support in $Y\times (0,1)\subset X$ that induce a basis of $\text{im}(\delta)\subset X$. As in the proof of Lemma \ref{lemma:mv2}, we may then choose $W$ so that
\[
H^2(X)=H_1\oplus H_2\oplus \text{im}(\delta)\oplus W,
\]
where the pairing restricted to $H_1\oplus H_2$ is non-degenerate, $H_1\oplus H_2$ is orthogonal to $\text{im}(\delta)\oplus W$, and there are closed forms $\eta_k''$ inducing a basis of $W$ such that
\[
    \int_X \eta_k' \wedge \eta_l' = \int_X \eta_k''\wedge \eta_l''= 0, \qquad \int_X \eta_k' \wedge \eta_l'' = \delta_{kl}.
\]
We observe that if $[\omega'],[\omega'']\in H^2(X)$ induce the same linear forms in $(H_1\oplus H_2\oplus \text{im}(\delta))^\ast$ via the pairing, then  $[\omega']\equiv [\omega'']\pmod{\text{im}(\delta)}$. Thus to establish \eqref{eq:firststep} it suffices to check \eqref{eq:weak2} for $\eta_j$ among the forms $\eta_{1,j}$, $\eta_{2,j}$, $\eta_k'$. As $\Phi_r(\omega)=\omega$ on $X_1$, where the support of $\eta_{1,j}$ lies, we have
\[
    \int_X f_r^\ast \Phi_r(\omega) \wedge \eta_{1,j} = \int_X \omega_c \wedge \eta_{1,j}.
\]
The support of $f_r^\ast \Phi_r(\omega)$ is disjoint from the supports of $\eta_{2,j}$ and $\eta_k'$,  and so pairs to give zero with these forms, just as is the case for $\omega_c$. Combined with $f_r^\ast  \mathbf{P}_r\Phi_r(\omega) - f_r^\ast \Phi_r(\omega)\to 0$ in $L^2$, this proves \eqref{eq:firststep}. Finally, Steps 1 and 2 combine to prove the proposition in the form \eqref{eq:propeqalt}.

\bibliographystyle{alpha}
\bibliography{main.bbl}


\Addresses

\end{document}